\documentclass[12pt]{article}

\usepackage{setspace}
\usepackage{amsfonts}
\usepackage{amsmath}
\usepackage{amssymb}
\usepackage[T1]{fontenc}
                            
\usepackage {eurosym}                                
\usepackage {pifont}                                 
\usepackage{setspace}
\usepackage {multicol}                               
\usepackage {times}
\usepackage{latexsym}
\usepackage{tabularx}
\usepackage{verbatim}

\usepackage [scaled=.90]{helvet}                     
\usepackage {courier}                                
\usepackage {graphicx}                               
\usepackage {array}                                  
\usepackage {longtable}                              
\usepackage{fancyvrb}

\newtheorem{theorem}{Theorem}
\newtheorem{lemma}{Lemma}

\newtheorem{remark}{Remark}

\newcommand{\NN}{{\mathbb N}}

\newcommand{\FF}{{\mathbb F}}

\newcommand{\bsT}{{\mathbf T}}
\newcommand{\bsc}{{\boldsymbol c}}
\newcommand{\supp}{{\rm supp}}
\newcommand{\bszero}{{\boldsymbol 0}}

\newenvironment{proof}{\begin{trivlist}

    \item[\hskip\labelsep{\it Proof.}]}{$\hfill\Box$\end{trivlist}}
\title{Explicit constructions of Vandermonde sequences using global function fields}

\author{Roswitha Hofer and Harald Niederreiter}

\date{}

\begin{document}

\setcounter{page}{1}
\maketitle

\begin{abstract}
The authors recently introduced so-called Vandermonde nets. These digital nets share properties with the well-known polynomial lattices. For example, both can be constructed via component-by-component search algorithms. A striking characteristic of the Vandermonde nets is that for fixed $m$ an explicit construction of $m\times m$ generating matrices over
the finite field $\FF_q$ is known for dimensions $s\leq q+1$.  
This paper extends this explicit construction in two directions. We give a maximal extension in terms of $m$ by introducing a construction algorithm for $\infty\times\infty$ generating matrices for digital sequences over $\FF_q$, which works in the rational function field over $\FF_q$. Furthermore, we generalize this method to global function fields
of positive genus, which leads to extensions in the dimension $s$.  

\end{abstract}

\noindent{\textbf{Keywords:} Low-Discrepancy Sequences, Digital Sequences, Generating Matrices, Global Function Field.}\\
\noindent{\textbf{MSC2010:} 11K31, 11K38, 14G15.}

\section{Introduction}
The authors \cite{HN} recently introduced a new familiy of digital $(t,m,s)$-nets called Vandermonde nets. In the present paper, we expand on this work. We consider chains of embedded Vandermonde nets or, equivalently, digital $(\bsT,s)$-sequences with generating matrices possessing a Vandermonde-type structure. This leads to new constructions of digital $(\bsT,s)$-sequences. We refer to the monograph \cite{DP} and the recent handbook article \cite{N13} for the necessary background on digital $(t,m,s)$-nets and digital $(\bsT,s)$-sequences. 

We start out by giving an explicit construction, using the rational function field $\FF_q(x)$ over the finite field $\FF_q$, of chains of embedded Vandermonde nets (we may equivalently speak of  Vandermonde sequences) which is based on methods and tools available in global function fields. Here the exponents appearing in the Vandermonde-type matrices are translated to multiples of selected divisors in Section~\ref{sec:2}. This relation to the Vandermonde structure is pointed out in Remark~\ref{rem:1}. Section~\ref{sec:3} gives a generalized version of the construction in global function fields with positive genus $g$. Finally, Section~\ref{sec:4} shows how, in the specific case where the entries of the generating matrices stem from local expansions at a rational place, it is possible to reach the usual best known quality parameter $t=g$, where $g$ is the genus of the underlying global function field. Concerning global function fields, we follow the notation and terminology in the book \cite{St}. 

Throughout this paper we use the convention that empty sums are considered to be $0$. Other statements or conditions depending on an empty index set (e.g. $j$ in the range $1\leq j\leq 0$) are disregarded.

\section{Embedded explicit construction in $\FF_q(x)$}\label{sec:2}

Let $s\geq 2$ be an integer and let $P_1,\ldots,P_s,P_\infty$ be distinct places of the rational function field $\FF_q(x)$, where $P_1,\ldots,P_s$ are rational and $P_\infty$ has an arbitrary degree $\mu$. Note $s\leq q+1$. (For example, choose $P_1$ as the infinite place of $\FF_q(x)$ and $P_2,\ldots,P_s$ as places corresponding to distinct monic linear polynomials $p_2(x),\ldots,p_s(x)$, say $p_2(x)=x+c_2=x,p_3(x)=x+c_3, \ldots, p_s(x)=x+c_s$, and let $P_\infty$ correspond to a further monic irreducible polynomial $p_\infty(x)$ of degree $\mu$.) The generating matrices $C^{(1)},\ldots,C^{(s)}\in\FF_q^{\NN\times \NN_0}$ are constructed row by row as follows. 
Choose 
$$\beta^{(1)}_j\in\mathcal{L}\left((j-1)(P_1-P_2)\right)\setminus\{0\}$$
and $$\beta^{(i)}_j\in\mathcal{L}\left(j(P_i-P_1)\right)\setminus\{0\}$$
for $j\geq 1$ and $i\in\{2,\ldots,s\}$. 
(In the example, one possible setting is $\beta^{(1)}_j=x^{j-1}$ and $\beta^{(i)}_j=1/(x+c_i)^j$ for $j\geq 1, 2\leq i\leq s$.)

\begin{lemma}\label{lem:1} Such $\beta_j^{(i)}$ exist for every $j\geq 1$ and $1\leq i\leq s$. 
We have \begin{enumerate}
\item $\nu_ {P_1}(\beta^{(1)}_j)=-j+1 $,
\item $\nu_{P_2}(\beta^{(1)}_j)=j-1 $,
\item $\nu_{P_i}(\beta^{(i)}_j)=-j $, and
\item $\nu_{P_1}(\beta^{(i)}_j)=j $
 \end{enumerate} for $j\geq 1$ and $2\leq i\leq s$. Furthermore, the system $\{\beta^{(i)}_j\}_{1\leq i\leq s, \, j\geq 1}$ is linearly independent over $\FF_q$. 
\end{lemma}

\begin{proof} Existence follows from the Riemann-Roch theorem \cite[Theorem 1.5.17]{St} since $\deg\left(P_i-P_h\right)=0$ for every choice of $1\leq i,h\leq s$, and therefore the dimension of $\mathcal{L}\left(P_i-P_h\right)$ is $1$. 
The values of the valuations follow directly from the construction of $\beta^{(i)}_j$ together with the fact that the degree of a principal divisor is $0$. The linear independence is now easily seen using the known values of the valuations at the places. 
\end{proof}

\begin{remark}\label{rem:1}{\rm
An alternative version would be to choose $\alpha_1\in\mathcal{L}\left(P_1-P_2\right)\setminus\{0\}$ and $\alpha_i\in\mathcal{L}\left(P_i-P_1\right)\setminus\{0\}$ for $2\leq i\leq s$ and set $\beta^{(1)}_j=\alpha_1^{j-1}$ and $\beta^{(i)}_j=\alpha_i^j$ for $j\geq 1$ and $2\leq i\leq s$, in order to obtain a Vandermonde structure of the matrix $C$ that determines the rows, i.e.,
$$C=\begin{pmatrix}1&\alpha_1 &\alpha_1^2 & \dots \\
\alpha_2&\alpha_2^2&\alpha_2^3 & \dots \\
\vdots&\vdots &\vdots & \dots \\
\alpha_s&\alpha_s^2&\alpha_s^3&\dots  \end{pmatrix}\in\FF_q^{s\times \NN}.$$
This conforms with the theory of Vandermonde nets developed in~\cite{HN}.
This choice is covered by the example at the beginning of this section, but is not adaptable to global function fields with positive genus. Note that the pole divisor of $\beta^{(2)}_1$ would be $P_2$, but this is impossible by the Weierstrass gap theorem \cite[Theorem 1.6.8]{St} which states that in the case of a positive genus, the integer $1$ has to be a gap number of the rational place $P_2$. }
\end{remark}

The $j$th row of $C^{(i)}$ is determined by the coefficients of the local expansion of $\beta^{(i)}_j$ at the place $P_\infty$, which is of the form $$\beta^{(i)}_j=\sum_{k=0}^\infty a^{(i)}_{j,k}z^k,$$
where $z=p_\infty(x)$ is a local parameter at $P_\infty$ and $a^{(i)}_{j,k}\in\{f(x)\in\FF_q[x] : \deg(f(x))< \mu\}$ for $k\geq 0$. Note that the construction ensures $\nu_{P_\infty}(\beta^{(i)}_{j})\geq 0$ for all $j\geq 1$ and $1\leq i\leq s$ and the set $\{f(x)\in\FF_q[x] : \deg(f(x))< \mu\}$ forms a vector space over $\FF_q$ with basis $1,x,\ldots,x^{\mu-1}$. (Note the linearity of the local expansion over $\FF_q$.) Now identify $a^{(i)}_{j,k}$ with a vector in $\FF_q^\mu$ with respect to the basis $1,x,\ldots,x^{\mu-1}$ and build the row $\bsc_j^{(i)}$ by concatenating these vectors. 

Before describing the quality of this construction in Theorem~\ref{thm:1} below, we note the following general fact about digital $(\bsT,s)$-sequences.

\begin{lemma}\label{lem:2}
Let $\mu,s\in\NN$ and let $C^{(1)},\ldots,C^{(s)}$ be generating matrices of a digital $(\bsT,s)$-sequence over $\FF_q$. If $\bsT(\ell\mu)\leq t$ for some $t \in \NN_0$ and every $\ell\in\NN$, then $\bsT:\NN\to\NN_0$ is given by $\bsT(m)=\min(m, t+r(m))$, where $r(m)$ is the least residue of $m$ mod $\mu$.
\end{lemma}

\begin{proof}
This easily follows from the definition of $\bsT:\NN\to\NN_0$ depending on the rank structure of $C^{(1)},\ldots,C^{(s)}$. 
\end{proof}

\begin{theorem}\label{thm:1}
The matrices $C^{(1)},\ldots,C^{(s)}$ constructed in this section generate a digital $(\bsT,s)$-sequence over $\FF_q$ with $\bsT(m)=r(m)$, where $r(m)$ is the least residue of $m$ mod $\mu$.
\end{theorem}

\begin{proof}
Because of Lemma \ref{lem:2}, it suffices to prove $\bsT(\ell\mu)=0$ for every $\ell\in\NN$, i.e., to ensure for every choice of $\ell\in\NN$ and $d_1,\ldots,d_s\in\NN_0$
 satisfying $1\leq d_1+\cdots+d_s\leq \ell\mu$ that the matrix formed by the row vectors
$$\bsc_j^{(i,m)}:= (c_{j,0}^{(i)},\ldots,c_{j,m-1}^{(i)})\in\FF_q^m$$
with $1\leq j\leq d_i$, $1\leq i\leq s$, and $m=\ell\mu$ has full row rank $d_1+\cdots+d_s$. 

Suppose that \begin{equation}\label{equ:1}
\sum_{i=1}^s\sum_{j=1}^{d_i}b_j^{(i)}\bsc_j^{(i,m)}=\bszero\in\FF_q^{m}
\end{equation}
for some $b_j^{(i)}\in\FF_q$. 
The linearity of the local expansion yields $$\beta:=\sum_{i=1}^s\sum_{j=1}^{d_i}b_j^{(i)}\beta_j^{(i)}=\sum_{k=0 }^\infty\left(\sum_{i=1}^s\sum_{j=1}^{d_i}b_j^{(i)}a^{(i)}_{j,k}\right)z^k.$$ 
Using \eqref{equ:1} together with the construction principle, we obtain 
$$\beta=\sum_{k=\ell }^\infty\left(\sum_{i=1}^s\sum_{j=1}^{d_i}b_j^{(i)}a^{(i)}_{j,k}\right)z^k.$$ 
Thus $\nu_{P_\infty}(\beta)\geq \ell$. 

We observe that
$$\beta_j^{(1)}\in\mathcal{L}((j-1)(P_1-P_2))\subseteq\mathcal{L}((d_1-1)P_1+d_2P_2+\cdots+d_sP_s)$$
for every $1\leq j\leq d_1$ and that  $$\beta_j^{(i)}\in\mathcal{L}(jP_i-jP_1)\subseteq\mathcal{L}(d_iP_i-P_1)\subseteq\mathcal{L}((d_1-1)P_1+d_2P_2+\cdots+d_sP_s)  $$
for $1\leq j\leq d_i,2\leq i\leq s$. 
This together with $\nu_{P_\infty}(\beta)\geq \ell$ implies that 
$$\beta\in\mathcal{L}((d_1-1)P_1+d_2P_2+\cdots+d_sP_s-\ell P_\infty)=:\mathcal{L}(D_{m,d_1,\ldots,d_s}).$$
Finally, we compute the degree of $D_{m,d_1,\ldots,d_s}$ and obtain 
$$\deg(D_{m,d_1,\ldots,d_s})=\left(\sum_{i=1}^sd_s\right)-1-\ell\mu\leq \ell\mu-1-\ell\mu=-1.$$
Thus $\beta=0$. Finally, the linear independence of the system $\{\beta^{(i)}_j\}_{1\leq i\leq s, \, j\geq 1}$, ensured by Lemma~\ref{lem:1}, implies that $b_j^{(i)}=0$ for $1\leq j\leq d_i,\,1\leq i\leq s$. 
\end{proof}

\begin{remark}{\rm 
In the case where $\mu=2$, the construction in this section provides a digital $(\bsT,q+1)$-sequence over $\FF_q$ with $\bsT(m)= 0$ for even $m$ and $\bsT(m)= 1$ for odd $m$. This construction is substantially simpler than the one introduced by Niederreiter and \"{O}zbudak in \cite{NiOe07} with the same function $\bsT$. 
 }
\end{remark}

\section{Generalization to $g>0$}\label{sec:3}
Let $F$ be a global function field with full constant field $\FF_q$ and genus $g > 0$. 
Let $s\geq 2$ be an integer and let $P_1,\ldots,P_s,P_\infty$ be distinct places of $F$, where $P_1,\ldots,P_s$ are rational and $P_\infty$ has an arbitrary degree $\mu$. 
We additionally use a positive auxiliary divisor $D$ of $F$ with $\deg(D)=2g$ and $\supp(D)\cap \{P_\infty,P_2,\ldots,P_s\}=\varnothing$ (e.g. $D=2gP_1$). We choose 
$$\beta^{(1)}_j\in\mathcal{L}\left(D+(j-1)P_1-(j-1)P_2\right)\setminus\mathcal{L}\left(D+(j-2)P_1-(j-1)P_2\right)$$
and $$\beta^{(i)}_j\in\mathcal{L}\left(D+jP_i-jP_1\right)\setminus\left(\mathcal{L}\left(D+jP_i-(j+1)P_1\right)\cup\mathcal{L}\left(D+(j-1)P_i-jP_1\right)\right)$$
for $j\geq 1$ and $2\leq i\leq s$.

\begin{lemma}\label{lem:3} Such $\beta_j^{(i)}$ exist for every $j\geq 1$ and $1\leq i\leq s$. 
We have \begin{enumerate}
\item $\nu_ {P_1}(\beta^{(1)}_j)=-(j-1)-\nu_{P_1}(D) $, 
\item $\nu_{P_h}(\beta^{(l)}_j)\geq 0 $, 
\item $\nu_{P_i}(\beta^{(i)}_j)=-j $, and
\item $\nu_{P_1}(\beta^{(i)}_j)=j -\nu_{P_1}(D)$
 \end{enumerate} for $j\geq 1$, for $2\leq i\leq s$, and for $2\leq h\leq s,\,1\leq l\leq s$ satisfying $ h\neq l$. Furthermore, the system $\{\beta^{(i)}_j\}_{1\leq i\leq s, \, j\geq 1}$ is linearly independent over $\FF_q$. 
\end{lemma}

\begin{proof} 
The values of the valuations follow directly from the construction of $\beta^{(i)}_j$ and from the special choice of $D$.
We observe that \begin{align}
\deg\left(\mathcal{L}\left(D+(j-1)P_1-(j-1)P_2\right)\right)&=2g,\label{eq:1}\\
\deg\left(\mathcal{L}\left(D+(j-2)P_1-(j-1)P_2\right)\right)&=2g-1, \label{eq:3}\\
\deg\left(\mathcal{L}\left(D+jP_i-jP_1\right)\right) &=2g,\label{eq:2}\\
\deg\left(\mathcal{L}\left(D+(j-1)P_i-jP_1\right) \right)&=2g-1,\label{eq:4}\\ 
\deg\left(\mathcal{L}\left(D+jP_i-(j+1)P_1\right) \right)&=2g-1, \label{eq:5}\\
 \left(\mathcal{L}\left(D+jP_i-(j+1)P_1\right)\cap\mathcal{L}\left(D+(j-1)P_i-jP_1 \right)\right) &\supseteq \{0\}.\label{eq:6}
\end{align}

Existence of $\beta_j^{(1)}$ for $j\geq 1$ follows directly from the Riemann-Roch theorem together with \eqref{eq:1} and \eqref{eq:3}. 
Existence of $\beta_j^{(i)}$ for $j\geq 1,2\leq i\leq s$ follows from
\begin{align*}
\left|\mathcal{L}\left(D+jP_i-jP_1\right)\right|-\left|\mathcal{L}\left(D+(j-1)P_i-jP_1\right) \right|-\left|\mathcal{L}\left(D+jP_i-(j+1)P_1\right) \right| & \\
+ \left|\mathcal{L}\left(D+jP_i-(j+1)P_1\right)\cap\mathcal{L}\left(D+(j-1)P_i-jP_1\right)\right| \geq  q^{g+1}-q^g-q^g +1& \geq 1 ,
\end{align*}
where we used the Riemann-Roch theorem together with \eqref{eq:2}, \eqref{eq:4}, \eqref{eq:5}, and \eqref{eq:6}. 

The values of the valuations can be used to see the linear independence of the $\beta_j^{(i)}$. 
\end{proof}

The $j$th row of $C^{(i)}$ is determined by the coefficients of the local expansion of $\beta^{(i)}_j$ at the place $P_\infty$, which is of the form $$\beta^{(i)}_j=\sum_{k=0}^\infty a^{(i)}_{j,k}z^k,$$
where $z$ is a local parameter at $P_\infty$ and the coefficients $a^{(i)}_{j,k}$ are chosen in a set of representatives, which ensures linearity of the local expansion and that each representative can be uniquely identified with a vector in $\FF_q^\mu$. Now build the $j$th row $\bsc_j^{(i)}$ of $C^{(i)}$ by concatenating these vectors. Note that the construction guarantees $\nu_{P_\infty}(\beta^{(i)}_{j})\geq 0$ for all $j\geq 1$ and $1\leq i\leq s$. 

\begin{theorem}
The matrices $C^{(1)},\ldots,C^{(s)}$ constructed in this section generate a digital $(\bsT,s)$-sequence over $\FF_q$ with $\bsT(m)=\min(m,2g+r(m))$, where $r(m)$ is the least residue of $m$ mod $\mu$.
\end{theorem}

\begin{proof}
Because of Lemma \ref{lem:2}, it suffices to prove $\bsT(\ell\mu)\leq 2g$ for every $\ell\in\NN$, i.e., to ensure for every choice of $\ell\in\NN$ satisfying $\ell\mu>2g$ and of $d_1,\ldots,d_s\in\NN_0$ satisfying $1 \le d_1+\cdots+d_s\leq \ell\mu-2g$ that the matrix formed by the row vectors
$$\bsc_j^{(i,m)}:= (c_{j,0}^{(i)},\ldots,c_{j,m-1}^{(i)})\in\FF_q^m$$
with $1\leq j\leq d_i,1\leq i\leq s$, and $m=\ell\mu$ has full row rank $d_1+\cdots+d_s$. 

Suppose that 
\begin{equation}\label{equ:8}\sum_{i=1}^s\sum_{j=1}^{d_i}b_j^{(i)}\bsc_j^{(i,m)}=\bszero\in\FF_q^{m}\end{equation}
for some $b_j^{(i)}\in\FF_q$. 
The linearity of the local expansion yields  $$\beta:=\sum_{i=1}^s\sum_{j=1}^{d_i}b_j^{(i)}\beta_j^{(i)}=\sum_{k = 0}^\infty \left(\sum_{i=1}^s\sum_{j=1}^{d_i}b_j^{(i)}a^{(i)}_{j,k}\right)z^k.$$
Using \eqref{equ:8} together with the construction principle, we obtain 
$$\beta=\sum_{k = \ell}^\infty \left(\sum_{i=1}^s\sum_{j=1}^{d_i}b_j^{(i)}a^{(i)}_{j,k}\right)z^k.$$
 Therefore $\nu_{P_\infty}(\beta)\geq \ell$.
 
Note that
$$\beta_j^{(1)}\in\mathcal{L}(D+(j-1)P_1-(j-1)P_2)\subseteq\mathcal{L}(D+(j-1)P_1)\subseteq\mathcal{L}(D+(d_1-1)P_1+d_2P_2+\cdots+d_sP_s)  $$ for $1\leq j\leq d_1$ and  $$\beta_j^{(i)}\in\mathcal{L}(D+jP_i-jP_1)\subseteq\mathcal{L}(D+d_iP_i-P_1)\subseteq\mathcal{L}(D+(d_1-1)P_1+d_2P_2+\cdots+d_sP_s)  $$
for $1\leq j\leq d_i,\, 2\leq i\leq s$. This together with $\nu_{P_\infty}(\beta)\geq \ell$ implies that 
$$\beta \in \mathcal{L}(D+(d_1-1)P_1+d_2P_2+\cdots+d_sP_s-\ell P_\infty)=:\mathcal{L}(D_{m,d_1,\ldots,d_s}).$$ 
We observe that $$\deg(D_{m,d_1,\ldots,d_s})\leq 2g+m-2g-1-m=-1$$ which implies that $\beta= 0$. Finally, the linear independence of the system $\{\beta^{(i)}_j\}_{1\leq i\leq s, \, j\geq 1}$, ensured by Lemma \ref{lem:3}, yields $b_j^{(i)} = 0$ for $1\leq j\leq d_i,\,1\leq i\leq s$. 
\end{proof}

\section{The case where $g>0$ and $P_\infty$ is rational}\label{sec:4}

Here a similar trick as used in the Xing-Niederreiter construction \cite{XiNi95} is available to save an additive term $g$ in the quality-parameter function $\bsT$.
Let $F$ be a global function field with full constant field $\FF_q$ and genus $g > 0$. 
Let $s\geq 2$ be an integer and let $P_1,\ldots,P_s,P_\infty$ be distinct rational places of $F$.

Analogously to the last section, we use a positive auxiliary divisor $D$ of $F$ with $\deg(D)=2g$ and $\supp(D)\cap \{P_\infty,P_2,\ldots,P_s\}=\varnothing$ (e.g. $D=2gP_1$) and choose 
$$\beta^{(1)}_j\in\mathcal{L}\left(D+(j-1)P_1-(j-1)P_2\right)\setminus\mathcal{L}\left(D+(j-2)P_1-(j-1)P_2\right)$$
and $$\beta^{(i)}_j\in\mathcal{L}\left(D+jP_i-jP_1\right)\setminus\left(\mathcal{L}\left(D+jP_i-(j+1)P_1\right)\cup
\mathcal{L}\left(D+(j-1)P_i-jP_1\right)\right),$$
for $j\geq 1$ and $2\leq i\leq s$. 
Note that the construction and the restriction on $D$ ensure $\nu_{P_\infty}(\beta^{(i)}_j)\geq 0$. Note that Lemma \ref{lem:3} is valid. 

Similarly to the Xing-Niederreiter construction, we construct a basis  $\{w_1,\ldots,w_{g}\}$ of the vector space $\mathcal{L}(D-P_1)$ with dimension $l(D-P_1)=g$ as follows.
By the Riemann-Roch theorem, we know the dimensions $l(D-P_{1})=g$ and $l(D-P_1-2gP_{\infty})=0$, hence there exist integers $0\leq n_1<\cdots<n_g<2g$ 
such that $$l(D-P_1-n_fP_{\infty})=l(D-P_1-(n_f+1)P_{\infty})+1 \text { for } 1\leq f\leq g.$$
Now choose $w_f\in\mathcal{L}(D-P_1-n_fP_{\infty})\setminus \mathcal{L}(D-P_1-(n_f+1)P_{\infty})$ to obtain the basis $$\{w_1,\ldots,w_g\} \text{ of  } \mathcal{L}(D-P_1).$$ Note that $\nu_{P_\infty}(w_f)=n_f$, $\nu_{P_1}(w_f)\geq 1-\nu_{P_1}(D)$, and $\nu_{P_i}(w_f)\geq 0$ for every $2\leq i\leq s$, $1\leq f\leq g$.

\begin{lemma}\label{lem:4}
The system $\{w_1,\ldots,w_g\}\cup \{\beta_j^{(i)}\}_{1\leq i\leq s, \, j\geq 1}$ is linearly independent over $\FF_q$. 
\end{lemma}

\begin{proof}
The linear independence of $\{w_1,\ldots,w_g\}$ and of $\{\beta_j^{(i)}\}_{j\geq 1}$ for every fixed $i=1,\ldots,s$ is obvious. 
Suppose that $$\sum_{f=1}^ga_fw_f +\sum_{i=1}^s\sum_{j=1}^vb_j^{(i)}\beta_j^{(i)}=0$$
for some $v\geq 1$ and $a_f,b_j^{(i)}\in\FF_q$. 
For a fixed $ h=2,\ldots, s$, we consider 
$$\sum_{j=1}^vb_j^{(h)}\beta_j^{(h)}=-\sum_{f=1}^ga_fw_f-\sum_{i=1\atop i\neq h}^s\sum_{j=1}^vb_j^{(i)}\beta_j^{(i)}.$$
We abbreviate the left-hand side by $\beta$. Now if $\beta \neq 0$, we know from Lemma \ref{lem:3} that 
$\nu_{P_h}(\beta)<0$. But the right-hand side satisfies $\nu_{P_h}(\beta)\geq 0$,  hence all coefficients on the left-hand side have to be $0$. We arrive at
$$\sum_{j=1}^vb_j^{(1)}\beta_j^{(1)}=-\sum_{f=1}^ga_fw_f.$$
We abbreviate the left-hand side by $\gamma$.
If there is a $b_j^{(1)}\neq 0$ for at least one $j\geq 1$, then the left-hand side yields $\nu_{P_1}(\gamma)\leq -\nu_{P_1}(D)$, but the right-hand side shows $\nu_{P_1}(\gamma)\geq  -\nu_{P_1}(D)+1$. Hence all $a_f$ and all $b_j^{(1)}$ have to be $0$. 
\end{proof}

The $j$th row of $C^{(i)}$ is determined by the coefficients of the local expansion of $\beta^{(i)}_j$ at the rational place $P_\infty$, which is of the form $$\beta^{(i)}_j=\sum_{k=0}^\infty a^{(i)}_{j,k}z_k,$$
where the coefficients $a^{(i)}_{j,k}\in\FF_q$ for $k\geq 0,j\geq 1,1\leq i\leq s$. The sequence $(z_k)_{k\geq 0}$ satisfies for $ k\in\NN_0\setminus\{n_1,\ldots,n_g\}$ that $\nu_{P_\infty}(z_k)=k$, and for $ k=n_f$ with $f\in\{1,\ldots,g\}$ that $z_k=w_f$. 
This preliminary construction yields the following sequence in $\FF_q$:
$$({a^{(i)}_{j,0}},\ldots,\widehat{a^{(i)}_{j,n_1}},a^{(i)}_{j,n_1+1},\ldots,\widehat{a^{(i)}_{j,n_g}},a^{(i)}_{j,n_g+1},\ldots)\in\FF_q^{\NN_0}.$$
After deleting the terms with the hat, we arrive at a sequence which serves as the $j$th row $\bsc_j^{(i)}$ of $C^{(i)}$, and we write
$$\bsc_j^{(i)}=(c_{j,0}^{(i)},c_{j,1}^{(i)},\ldots)\in\FF_q^{\NN_0}.$$

\begin{theorem}
The matrices $C^{(1)},\ldots,C^{(s)}$ constructed in this section generate a digital $(t,s)$-sequence over $\FF_q$ with $t=g$. 
\end{theorem}

\begin{proof}
We prove for every integer $m>g$ and every choice of nonnegative integers $d_1,\ldots,d_s$ satisfying $1 \le d_1+\cdots+d_s\leq m-g$ that the matrix formed by the row vectors
$$\bsc_j^{(i,m)}:= (c_{j,0}^{(i)},\ldots,c_{j,m-1}^{(i)})\in\FF_q^m$$
with $1\leq j\leq d_i,1\leq i\leq s$, has full row rank $d_1+\cdots+d_s$.

Choose $b_j^{(i)}\in\FF_q$ for $1\leq j\leq d_i,1\leq i\leq s$, satisfying \begin{equation}\label{equ:9}\sum_{i=1}^s\sum_{j=1}^{d_i}b_j^{(i)}\bsc_j^{(i,m)}=\bszero\in\FF_q^{m}.\end{equation}
The linearity of the local expansion implies that $$\beta:=\sum_{i=1}^s\sum_{j=1}^{d_i}b_j^{(i)}\beta_j^{(i)}-\underbrace{\sum_{i=1}^s\sum_{j=1}^{d_i}b_j^{(i)}\sum_{f=1}^g a^{(i)}_{j,n_f}w_f}_{=:\alpha}=\sum_{{k=0}\atop{k\neq n_1,\ldots,n_g}} ^\infty \left(\sum_{i=1}^s\sum_{j=1}^{d_i}b_j^{(i)}a^{(i)}_{j,k}\right)z_k.$$ 
In view of the construction algorithm and \eqref{equ:9}, we obtain
$$\beta=\sum_{{k=m+g}} ^\infty \left(\sum_{i=1}^s\sum_{j=1}^{d_i}b_j^{(i)}a^{(i)}_{j,k}\right)z_k.$$ 
Therefore $\nu_{P_\infty}(\beta)\geq m+g$. 

We observe that 
$$\alpha\in \mathcal{L}(D-P_1)\subseteq\mathcal{L}(D+(d_1-1)P_1+d_2P_2+\cdots+d_sP_s) ,$$
$$\beta_j^{(1)}\in\mathcal{L}(D+(j-1)P_1-(j-1)P_2)\subseteq\mathcal{L}(D+(d_1-1)P_1+d_2P_2+\cdots+d_sP_s)  $$ for $1\leq j\leq d_1$, and  $$\beta_j^{(i)}\in\mathcal{L}(D+jP_i-jP_1)\subseteq\mathcal{L}(D+d_iP_i-P_1)\subseteq\mathcal{L}(D+(d_1-1)P_1+d_2P_2+\cdots+d_sP_s)  $$
for $1\leq j\leq d_i,\,2\leq i\leq s$. 
This together with $\nu_{P_\infty}(\beta)\geq m+g$ implies that
$$\beta \in \mathcal{L}(D+(d_1-1)P_1+d_2P_2+\cdots+d_sP_s-(m+g)P_\infty)=:\mathcal{L}(D_{m,d_1,\ldots,d_s}).$$ 
We compute the degree of $D_{m,d_1,\ldots,d_s}$ and obtain $$\deg(D_{m,d_1,\ldots,d_s})\leq 2g+m-g-1-(m+g)=-1,$$ which entails $\beta= 0$. Finally, the linear independence of the system $\{w_1,\ldots,w_g\}\cup \{\beta_j^{(i)}\}_{1\leq i\leq s, \, j\geq 1}$ yields $b_j^{(i)}=0$ for $1\leq j\leq d_i,\,1\leq i\leq s$.  
\end{proof}

\vspace{1cm}
\noindent{Roswitha Hofer, Institute of Financial Mathematics, University of Linz, Altenbergerstr.
69, A-4040 Linz, Austria; email: roswitha.hofer@jku.at} \\

\noindent{Harald Niederreiter, Johann Radon Institute for Computational and Applied
Mathematics, Austrian Academy of Sciences, Altenbergerstr. 69, A-4040
Linz, Austria, and Department of Mathematics, University of Salzburg,
Hellbrunnerstr. 34, A-5020 Salzburg, Austria; email: ghnied@gmail.com}

\end{document}